\documentclass[12pt]{article}

\usepackage[dvipdfmx]{graphicx}
\usepackage{tikz-cd}
\usepackage{mymacro}
\usepackage{typearea}
\usepackage{cleveref}

\renewcommand{\Sh}{\mathrm{Sh}}

\renewcommand{\Mod}{\mathrm{Mod}}

\newcommand{\musupp}{\mu\mathrm{supp}}

\newcommand{\Loc}{\mathrm{Loc}}

\title{On the generic existence of WKB spectral networks/Stokes graphs}
\date{}
\author{Tatsuki Kuwagaki}
\begin{document}

\maketitle
\begin{abstract}
We prove the generic existence of spectral networks for a large class of spectral data. 
\end{abstract}

\section{Introduction}
Let $C$ be a compact Riemann surface and $\Phi$ be a meromorphic quadratic differential on $C$. For $\theta\in S^1$, the equation $\Image \sqrt{\Phi}=0$ gives a codimension 1 foliation on $C$. The set of leaves passing through the zero set of $\Phi$ is called, the Stokes graph of $\Phi$ in the literature of exact WKB analysis. It is a fundamental observation of Voros~\cite{Voros}, and later proved by Koike--Sch\"afke~\cite{KoikeShaefke} (see also \cite{Takei, Nikolaev}) that, for a large class of $\Phi$, we can analytically lift the all order formal WKB solutions of a deformation quantization of the spectral curve of $\Phi$ on the complement of the Stokes graph for generic $\theta$. By assigning connection matrices on each curve on the graph, we can completely package the solution to the connection problem of the exact WKB solutions.

Independently, the same structure was discovered by Gaiotto--Moore--Neitzke~\cite{GMNWKB} in the study of a certain class of $\cN=2, d=4$ quantum field theories called class $\cS$ theories. In their frame work, the space of meromorphic quadratic differentials of a given pole type is interpreted as the Coulomb branch of the theory. Fix $\theta\in S^1$ and a point of the Coulomb branch determines a ``phase" of the theory, and each curve passing through $z\in C$ in the Stokes graph is a BPS-state of the theory in the presence of a defect on $z$. The reason why the Stokes graph is important in physics is now evident: Knowing the contents of the BPS-spectrum is a first step to understand the theory. Similar to the case of exact WKB analysis, one can assign BPS multiplicities on the curves of the diagram. In their terminology, Stokes graph is called (WKB) 
{\em spectral network}.

\begin{remark}
    In the body of the paper, we use the term {\em Stokes graph} rather than WKB spectral network.
\end{remark}

To generalize exact WKB analysis to ``higher order differentials" is a long desired story, but there are several obstacles to establish such a theory. The first obstacle is that we do not know how to define Stokes graphs for the higher order case. It is known that a naive generalization does not work, as observed by Berk--Nevins--Roberts~\cite{BNR}. The next obstacle is the resummation problem, but it is even more harder: We cannot even formulate a conjecture without knowing the Stokes graph.

One of the subjects of this paper is to provide a generic existence theorem of Stokes graph/spectral network for a large class of pole data. 

Our approach is based on Gaiotto--Moore--Neitzke's work~\cite{GMNspec}. In their work, they proposed an inductive construction of spectral networks using mass filtration and claimed that their construction generically work. Although their ideas seem to be roughly correct, we believe that to transform it to a mathematically rigorous statement, several works have to be done. Our work here is to provide mathematically rigorous setup/notions/proofs to make their construction actually work.

Now we would like to state our results more precisely. Let $D$ be a finite subset of $C$ and $K$ be a fixed positive number. We take a point $\Phi$ of the meromorphic Hitchin base $\bigoplus_{i=1}^KH^0(C, \cK_C^{\otimes i}(*D))$ where $\cK_C$ is the sheaf of meromorphic sections of the canonical sheaf with poles in $D$. In this work, we assume $\Phi$ is \emph{strongly GMN} (Definition~\ref{def:GMN}) which is a higher-order generalization of (a little bit stronger version of) GMN condition appeared in the work of Bridgeland--Smith~\cite{BridgelandSmith}.

For a point $\Phi$, one can associate a complex 1-dimensional submanifold called the {\em spectral curve} $L$ of $\Phi$ in the cotangent bundle $T^*C$. The restriction of the projection $\pi\colon T^*C\rightarrow C$ to $L$ gives a branched $K$-fold covering map. We define the set of turning points to be the branching values of the covering map, which generalizes the 2nd order case. 

Around a point in $C$ outside the branching values, take a pair of sheets of $\pi|_L$. Then the difference between the restriction of the holomorphic Liouville form $\lambda$ to the sheets gives a holomorphic 1-form on $L$. Then the reality condition of the 1-form times $e^{2\pi i\theta}$ gives a local foliation of $C$. A leaf of the foliation is called a $\theta$-preStokes curve.

Roughly speaking, a $\theta$-Stokes graph is a set of $\theta$-preStokes curves first emanating from the set of turning points, and then evolving again and again from the collisions of the curves. We say a $\theta$-Stokes graph is unobstructed if it does not contain any 4d BPS states (Definition~\ref{def:unobs}), which is a tree whose edges are $\theta$-preStokes curves and the vertices are turning points.
\begin{theorem}
Suppose $\Phi$ is strongly GMN. There exists a dense set of $S^1$ such that unobstructed $\theta$-Stokes graph exists for any $\theta$ in the subset.
\end{theorem}
We can strengthen this theorem to a statement with compatible wall-crossing data as follows:  Wall-crossing data is a set of pairs of a preStokes curve and wall-crossing matrix (or connection matrix in the context of exact WKB analysis). Associated to it, we have wall-crossing matrix. The compatibility means the vanishing of undesired monodromies.

As Gaiotto--Moore--Neitzke observed, Stokes graphs/spectral networks are filtered by BPS mass. An initial wall-crossing data is a wall-crossing data defined at the level of $\text{mass}=0$.
Here is our another main theorem.
\begin{theorem}
Suppose $\Phi$ is strongly GMN. Let $I$ be an open subset of $S^1$. Give an initial wall-crossing data. Then there exists a dense subset of $I$ such that a Stokes graph with compatible wall-crossing data exists for any $\theta\in I$.
\end{theorem}
Our proof is based on Gaiotto--Moore--Neitzke's argument, which is very close to the construction of scattering diagrams~\cite{KontsevichSoiblemanAffine,GrossSiebert}: We first start with the preStokes curves emanating from the branching values of $L$. Then, we possibly have anomalous monodromy around the collisions. Then we add new Stokes curves with wall-crossing factors to cancel the anomalous monodromies. Then the new Stokes curves can have collisions and produce anomalous monodromies again. Then we repeat the procedures.

This prescription is very convincing, but there are potential problems: (1) Possibly there are situations where we cannot cancel anomalous monodromies, (2) the set of collisions/preStokes curves in the network is possibly non-discrete/dense for which we cannot define local monodromies.

We can observe that (1) only happens when there are 4d BPS states which does not happen for generic $\theta$. For (2), Gaiotto--Moore--Neitzke used the mass filtration to draw diagrams.

To make this procedure mathematically rigorous, our key machineries/observations/proofs are the following:
\begin{enumerate}
    \item Associated quadratic differential: To control the behavior of preStokes curves, we associate a quadratic differential on a covering of $C$. Then we can use the classical theory of quadratic differentials~\cite{Strebel} to control preStokes curves.
    \item Novikov ring: To describe local monodromies in a way compatible with the mass filtration, it is better to consider them as defined over the Novikov ring. This is parallel to the use of the formal power series ring in the theory of scattering diagram.
    \item Stoke trees: To perturb $\theta$ to collapse 4d BPS states, we have to treat a family of Stokes graphs/spectral networks, which is not well-behaved. We instead treat it as a union of Stokes trees. This interpretation is motivated by \cite{GMNspec} and a folklore conjectural interpretation of spectral networks in terms of holomorphic disks bounded by the spectral curve. For the 2nd order case, see Nho~\cite{nho2024family} for a related result. The expected interpretation in the present setup will be shown in \cite{IKO}.
    \item Gromov compactness: In our construction, we take mass cut-off. Step by step, we can take our mass cut-off greater than the previous cut-off. We have to prove the cut-off eventually goes to $+\infty$. One can prove this by careful estimates of mass.
\end{enumerate}
With these machineries, we can prove the existence of spectral networks.

\begin{remark}
    A part of the argument of the proof of this theorem has already appeared in the author's previous work~\cite[v3]{WKBkuw}. However a result provided in the paper is proved under very strong assumptions, which has been never checked for a large class of examples. In the referee process of the paper, the anonymous referee suggested the author to remove the section from the paper. The result of a large enhancement of the section is this paper.
\end{remark}

\begin{remark}
    From the side of exact WKB analysis, the theory of virtual turning points has been developed~\cite{Virtual}. Based on this theory, Honda~\cite{Honda} proposed an approach to a general construction of Stokes graph under certain conditions. Although some of the ideas here are parallel to his work, the statement of his theorem is very different from ours, in particular, it does not prove any generic existence. See the original article for more information.
\end{remark}

Now we give several applications of our result.

The first one is an exact formulation of exact WKB conjecture, which will be provided in Section~\ref{WKBconjecture}.

The second one is the Gromov compactness theorem for 4d BPS states. In this paper, a 4d BPS-state is a Stokes tree ended in turning points. 
\begin{theorem}
For any $\theta\in S^1$, any strongly GMN Lagrangian $L$, and any $M\geq 0$, the number of the BPS states with mass less than $M$ is finite.
\end{theorem}
The proof is a minor modification of the proof of the main theorem of this paper. Conjecturally, this statement is a Morse-theoretic counterpart of the Gromov compactness theorem of holomorphic disks bounded by the spectral curve.

The third application is a construction of sheaf quantization of the spectral curve as conducted in \cite{WKBkuw} in the 2nd order case. Again, let $L$ be a strongly GMN Lagrangian. We specify a brane structure of $L$. For $\theta \in S^1$, we set $L_\theta:=e^{2\pi i\theta}\cdot L$, which also carries an induced brane structure.

If $L_\theta$ is unobstructed, it defines an object of (non-curved) Fukaya category of $T^*C$. By the Solomon--Verbitsky theorem~\cite{SolomonVerbitsky}, the set of unobstructed $\theta$'s are dense in $S^1$. 

On the other hand, there is a conjectural sheaf-theoretic model of nonexact Fukaya category of $T^*M$, called the category of sheaf quantizations~\cite{Tam, WKBkuw, IK, KPS}. The following theorem is a sheaf-theoretic counterpart of a version of the Solomon--Verbitsky theorem.
\begin{theorem}
    There exists a dense set $D$ of $S^1$ such that there exists a sheaf quantization of $L_\theta$ for any $\theta\in D$.
\end{theorem}
The proof goes as follows: A brane structure gives an initial spectral network. Then the main theorem gives a spectral network. Now, as in \cite{WKBkuw}, one can glue up a sheaf quantization by using wall-crossing factors. From sheaf quantizations, one can also construct a local system over the base, which is nonabelianization in the sense of \cite{GMNspec} in our setup.

\subsection*{Acknowledgment}
I would like to thank Akishi Ikeda, Hiroshi Iritani, Tsukasa Ishibashi, Kohei Iwaki, and Hiroshi Ohta for discussions on several topics related to the subject of this paper. I also would like to thank the members of GHKK reading seminar (Tsukasa Ishibashi, Shunsuke Kano, Yuma Mizuno, Hironori Oya) where I learned the notion of scattering diagram. This work is supported by JSPS KAKENHI  Grant Numbers JP22K13912 and JP20H01794.

\section{Setup}
In this section, we introduce our main objects. There are two ways to describe it by using differentials and by using Lagrangians. These two descriptions are equivalent and both useful. In the following section, we will use these two descriptions interchangeably.
\subsection{Differential language}
Let $C$ be a compact Riemann surface and $D$ be a nonempty finite subset on $C$. Let $\cK_C(*D)$ be the sheaf of meromorphic differentials with possible poles in $D$.
\begin{definition}
Fix a positive integer $n$. A ($K$-)spectral data is an element
\begin{equation}
    \Phi=\bigoplus_{i=1}^K\Phi_i\in \bigoplus_{i=1}^K H^0(C, \cK_C^{\otimes i}(*D)).
\end{equation}
\end{definition}

Let $\Phi$ be a spectral data. 
Let $z$ be a local coordinate of $C$. Then $\Phi$ can be locally written as
\begin{equation}
    \Phi=\sum_{i=1}^K \phi_i(z) dz^{\otimes i}
\end{equation}
Consider the symbol $\sigma(\Phi):=\zeta^K-\sum_{i=1}^K\phi_i\zeta^{K-i}$. We set
\begin{equation}
    L_\Phi:=\lc (z, \zeta)\in T^*(C\bs D)\relmid \sigma(\Phi)=0\rc
\end{equation}
where $\zeta$ is identified with the cotangent of $z$. This is independent of the choice of the local coordinate. The resulting 1-dimensional complex subvariety of $T^*C$ is called the {\em spectral curve of $\Phi$}.

We denote the defining projection $T^*C\rightarrow C$ by $\pi$. 
\begin{definition}
We say a branching point of $\pi|_{L_\Phi}$ is simple if the branching order is minimal.

\end{definition}
Note also that the set of the branching values is finite by the compactness of $C$.

\subsection{Symplectic language}
Let $z$ be a local coordinate of $C$ and $\zeta$ be the associated cotangent coordinate. Then the 1-form $\zeta dz$ on $T^*C$ does not depend on the choice of the local coordinate. The resulting global 1-form is called the Liouville form $\lambda$. The 2-form $\omega:=d\lambda$ gives a canonical symplectic structure of $T^*C$. 

Let $L$ be a holomorphic Lagrangian submanifold (equivalently, 1-dimensional smooth complex submanifold).
\begin{definition}
    We say $L$ is meromorphic if there exists $\Phi$ such that $L=L_\Phi$.
\end{definition}

We can recover $\Phi$ from $L_\Phi$ as follows. For each contractible open subset $U$ in $C\bs D$, we have $\pi|^{-1}_L(U)=U_1\sqcup \cdots \sqcup U_K$ where each $U_i$ is isomorphic to $U$ through $\pi|_L$. For each $i\in [K]=\lc 1,...,K\rc$, we have the 1-form on $U$ by pulling back $\lambda_i:=\lambda|_{U_i}$. Then we set
\begin{equation}
    \zeta^K-\sum_{i=1}^K\phi_i\zeta^{K-i}:=\prod_{i=1}^K(\zeta-\lambda_i).
\end{equation}
Then $\bigoplus_{i=1}^K\phi_i$ recovers $\Phi$.

For this reason, we will use $\Phi$ and $L$ interchangeably in the following sections.

\section{Trajectory structure}
We fix a meromorphic Lagrangian $L$. We denote the restriction of the projection $\pi\colon T^*C\rightarrow C$ to $L$ by $\pi_L$.

\subsection{PreStokes curves and trajectory}

\begin{definition}
Let $\gamma$ be a smooth immersed curve in $C\bs D$. A type structure of $\gamma$ is an ordered pair $\frakt=(s_1, s_2)$ of sections of $\gamma^*L$.
\end{definition}
For $\frakt=(s_1, s_2)$, we set $\frakt^{op}=(s_2, s_1)$.

\begin{definition}
We say an immersed curve with a type structure $(\gamma, (s_1, s_2))$ is a preStokes curve of type $(s_1, s_2)$ if the following hold
    \begin{enumerate}
        \item $s_1, s_2$ does not pass branching points.
        \item $\Im (e^{-2\pi i\theta}(\lambda_1-\lambda_2))=0$  where $\lambda_i:=s_i^*\lambda$.
    \end{enumerate}
We always orient a preStokes curve in a way that $\int\Re(e^{-2\pi i\theta}(\lambda_1-\lambda_2))$ is increasing.
\end{definition}

\begin{definition}
    We say a branching point (resp. value) is of type $\frakt=(s_1, s_2)$ if it is a simple branching point (resp. value) of $\pi_L$ where $s_1$ and $s_2$ merge. Note that a branching point is of type $\frakt=(s_1, s_2)$ is also of type $\frakt^{op}$.
\end{definition}

\begin{remark}
    In the literature of exact WKB analysis, branching value is called {\em turning point}.
\end{remark}

\begin{definition}
    A maximal preStokes curve is called a {\em trajectory}.
\end{definition}

\begin{definition}
\begin{enumerate}
    \item A saddle trajectory is a trajectory of type $(ij)$ connecting branching values of the same type.
    \item A closed trajectory is a periodic trajectory.
    \item A recurrent trajectory is a trajectory whose closure has nonempty interior.
\end{enumerate}
\end{definition}

The followings are basic facts.
\begin{lemma}[\cite{Strebel}]
For each simple branching value of type $(ij)$, three preStokes curves of type $(ij)$ emanate.
\end{lemma}

We later use the following lemma:
\begin{lemma}\label{lem:discreteintersection}
    Let $\gamma_1, \gamma_2$ be two preStokes curves which are not recurrent and distinct. The intersections between $\gamma_1$ and $\gamma_2$ are discrete.
\end{lemma}
\begin{proof}
    Since $\gamma_1$ and $\gamma_2$ are real analytic curve, they coincide if they have accumulating intersection points.
\end{proof}

\subsection{The case of quadratic differentials}
In this section, we consider the case when $\Phi\in H^0(C, \cK_C^{\otimes 2}(*D))$. In this case, the theory is classical. We refer to Strebel~\cite{Strebel} and Bridgeland--Smith~\cite{BridgelandSmith}. We use a little stronger version of the GMN condition in \cite{BridgelandSmith}.
\begin{definition}
A quadratic differential $\Phi$ is strongly GMN if the following holds:
\begin{enumerate}
    \item Any branch point is simple.
    \item $\Phi$ has at least one branch point.
    \item Any pole is with order $\geq 2$.
    \item $\Phi$ has at least one pole.
\end{enumerate}
\end{definition}

\begin{lemma}[{\cite[Lemma 3.1]{BridgelandSmith}}]\label{lem:norecurrence}
    If $\Phi$ is strongly GMN, there exists a nonempty open subset $V\subset S^1$ such that $\Phi$ does not admit closed and recurrent trajectories for $\theta\in V$.
\end{lemma}

\subsection{Associated quadratic differential}
To deal with $\Phi$, it is convenient to use the knowledge from the theory of quadratic differentials. For this purpose, we introduce an associated quadratic differential as follows.

\begin{example}
    To explain the following construction, we first deal with an example of $K=2$; $\Phi_1\oplus \Phi_2$. For the defining equation of the spectral curve $\zeta^2-\Phi_1\zeta-\Phi_2$, we take the square completion $(\zeta-\Phi_1/2)+\Phi_2-\Phi_1^2/4$. It is easy to observe that any preStokes curve defined by $\Phi$ is a preStokes curve of the quadratic differential $\Phi_2-\Phi_1^2/4$, and vice verca. 
\end{example}
In the following, we generalize this example to higher order cases. 
Let $B(\Phi)\subset C$ be the set of branching values of $\Phi$. Take a contractible open subset $U$ in $C\bs (D\cup B(\Phi))$. Then we have $\pi^{-1}(U)=U_1\sqcup \cdots \sqcup U_K$ where each $U_i$ is isomorphic to $U$ through $\pi$. For each $i\in [K]=\lc 1,...,K\rc$, we have the 1-form on $U$ by pulling back $\lambda_i:=\lambda|_{U_i}$, which is again denoted by $\lambda_i$. Let $\Delta$ be the diagonal subset of $\Sym^2[K]$. For each $(ij)\in \Sym^2[K]\bs \Delta$, we set
\begin{equation}
    \lambda_{ij}:=(\lambda_i-\lambda_j)^2,
\end{equation}
which is a quadratic differential on $U$. Consider the analytic continuation of $\lambda_{ij}$ over $C\bs (D\cup B(\Phi))$, and we denote the resulting Riemann surface by $C_{ij}'$.

Note that each sheet of $C_{ij}$ is parametrized by $(kl)\in \Sym^2[K]\bs \Delta$. For each point of $B(\Phi)$, we can fill out $C_{ij}'$ by one of the following:
\begin{enumerate}
    \item if a sheet is corresponding to $(kl)$ which is the type of the branch, then we fill out it by a disk. Then $\lambda_{ij}$ is extended by zero.
    \item if a sheet is corresponding to $(kl)$ which is not the type of the branch, then we fill out it by a branching disk. Then $\lambda_{ij}$ is extended.
\end{enumerate}
We denote the resulting branched covering of $C\bs D$ by $C_{ij}$. The following is obvious from the above construction.
\begin{lemma}
    The Riemann surface $C_{ij}$ is a finite covering of $C\bs D$.
\end{lemma}

As a result, we get a finite covering $\pi_{\widetilde C}\colon \widetilde{C}:=\bigcup_{(i,j)\in \Sym^2[K]\bs \Delta}C_{i j}\rightarrow C\bs D$ and a quadratic differential $\widetilde \Phi$ on $\widetilde C$.

For our purpose, either of construction works well. The following is also obvious.
\begin{lemma}
    Any preStokes curve of $\Phi$ is the projection of a preStokes curve of $\widetilde\Phi$.
\end{lemma}

\begin{remark}
Alternative way to get a similar object is as follows: We consider the fiber product $L\times_CL$ using $\pi_L$. We denote the projections by $\pi_i\colon L\times_CL\rightarrow L$ ($i=1,2$).
This is again a finite branched covering of $C$. Note that the diagonal $\Delta\subset L\times_CL$ is a connected component. On the complement $\Delta^c$, we consider the quadratic differential $(\pi_1^{*}\lambda_1-\pi_1^{*}\lambda_2)^2$. It descends to a quadratic differential $\widetilde\Phi'$ on $\widetilde C':=\Delta^c/(\bZ/2)$. We have a branched covering map $\widetilde C'\to C\bs D$, and any preStokes curve of $\Phi$ is the projection of a preStokes curve of $\widetilde\Phi '$.
\end{remark}

\subsection{Global structure}
To have tame behavior of global trajectories, we assume a higher-order analogue of GMN condition.
\begin{definition}[Higher order GMN condition]\label{def:GMN}
We say $\Phi$ is strongly GMN if the followings hold:
\begin{enumerate}
    \item $\widetilde \Phi$ is strongly GMN, and
    \item $\widetilde \Phi$ has poles on $D$.
\end{enumerate}
\end{definition}

\begin{remark}
    For our main results, one can remove the second condition easily by perturbing $\theta$ more. We assume here just for simplicity of the notation.
\end{remark}

Here we give a sufficient condition to be strongly GMN. For this purpose, we recall the Hukuhara--Levelt--Turittin theorem.
\begin{definition-lemma}
    Let $(\cE, \nabla)$ be a meromorphic flat connection over $\bC\{z\}$. After the base change to $\bC[[z]]$, there exists an isomorphism
    \begin{equation}
        (\cE, \nabla)\cong \bigoplus_i(\cE(f_i)\otimes R_i, \nabla_i\otimes \nabla_i')
    \end{equation}
    where $f_i$ is a finite Puiseux series, $(\cE(f_i), \nabla_i)=(\cO\{z\}, d-df_i)$, and $(R_i, \nabla_i')$ is a regular connection. The set $\{(f_i, \rank R_i)\}$ is called the formal type of $\nabla$.
\end{definition-lemma}
For a point $z\in D$, take a local neighborhood, and consider a differential equation 
\begin{equation}
    (\partial^K-\sum_{i=1}^K\Phi_i\partial^{K-i})\psi=0
\end{equation}
associated to $\Phi$. Viewing this as a differential system, we define the formal type of $\Phi$ as the formal type of this equation.

\begin{lemma}
    $\Phi$ is strongly GMN if the following holds:
\begin{enumerate}
    \item Any branching point is simple.
    \item There exists at least one branching point for each connected component of the spectral curve.
    \item For any $i$, $f_i$ of the formal type has pole order $\geq 1$ at any pole.
    \item For any $i$, $\rank R_i=1$. 
\end{enumerate}
\end{lemma}
\begin{proof}
We would like to check that $\widetilde \Phi$ satisfies Definition~\ref{def:GMN}.
    The conditions 1 and 2 in the above imply the conditions 1 and 2 of Definition~\ref{def:GMN}. The conditions 3 and 4 in the above implies 3 of Definition~\ref{def:GMN}. The condition 3 implies 4 of Definition~\ref{def:GMN}.
\end{proof}

\begin{lemma}
    If $\Phi$ is strongly GMN, it does not have closed trajectories and recurrent trajectories for generic $\theta$.
\end{lemma}
\begin{proof}
    This is by Lemma~\ref{lem:norecurrence}.
\end{proof}

\section{Stokes graph (a.k.a. WKB spectral network)}
In this section, we first define the Stokes graph (a.k.a. WKB spectral network).

\subsection{Stokes trees}
A tree is a connected acyclic graph whose valency at each interior vertex is $\geq 3$. A rooted tree is a tree with a distinguished exterior vertex. For a rooted tree, the distinguished vertex is called the root vertex. The edge connected to the root vertex is called the root edge. A leaf vertex is an exterior vertex which is not the root vertex. A leaf edge is the edge connected to a leaf vertex. We always orient a rooted tree toward the root.

We first define the notion of a Stokes tree. We fix a meromorphic Lagrangian $L$ and $\theta$. 
\begin{definition}[Open Stokes tree, or 2d BPS state]
An open Stokes tree is the following: an immersion $\iota_\cT\colon \cT\rightarrow C$ of a finite rooted tree  $\cT$ with type structures $\frakt_e$ on the interior of each edge $e$ of $\cT$ such that 
\begin{enumerate}
    \item Each edge $e$ is a preStokes curve of type $\frakt_e$.
    \item The orientation of each edge as a preStokes curve is the same as the one induced from the orientation of $\cT$.
    \item The root vertex is in $D$. The restriction $\iota|_{\cT\bs \{root\}}\subset C\bs D$.
    \item Any leaf edge of type $\frakt$ has its leaf vertex is in the branching value of type $\frakt$.
    \item Any interior vertex is in $C\bs (D\cup B(\Phi))$. Here $B(\Phi)$ is the set of branching values.
    \item Each interior vertex has the cyclic order induced from the orientation of $C$. Suppose $v$ is an interior vertex and consider the ordered set of edges $e_1,...,e_k$ on $v$ with rooted $e_1$. Then the condition is the following: Each $e_i$ is a preStokes curve of type $(i-1,i)$ for $i\neq 1$. For $i=1$, it is a preStokes curve of type $(k1)$.
\end{enumerate}
We always consider each edge of an open Stokes tree as a preStokes curve is of the type whose orientation is compatible with the orientation of the rooted tree. We always consider open Stokes trees up to reparametrization.
\end{definition}

\begin{definition}[Closed Stokes tree, or 4d BPS state of genus $0$]
A closed Stokes tree is the following: a finite tree $\cT$ immersed in $C$ with a type structure $\frakt_e$ on the interior of each edge $e$ of $\cT$ satisfying the following:
\begin{enumerate}
    \item Each edge $e$ is a preStokes curve of type $\frakt_e$.
    \item Any exterior vertex of type $\frakt$ has its leaf vertex is in the branching value of type $\frakt$.
    \item Any interior vertex is in $C\bs (D\cup B(\Phi))$. 
    \item Each interior vertex has the cyclic order induced from the orientation of $C$. Suppose $v$ is an interior vertex and consider the ordered set of edges $e_1,...,e_k$ on $v$ with rooted $e_1$. Then the condition is the following: Each $e_i$ is a preStokes curve of type $(i-1,i)$ up to opposition.
\end{enumerate}
We always consider each edge of a rooted Stokes tree as a preStokes curve is of the type whose orientation is compatible with the orientation of the rooted tree.  We always consider closed Stokes trees up to reparametrization.
\end{definition}

\begin{definition}
\begin{enumerate}
    \item  The Stokes graph $\frakS_{L, \theta}$ of $L$ at $\theta$ is the set of open and closed Stokes trees. We denote the subset of open Stokes trees by $\frakO_{L, \theta}$, and the set of closed Stokes trees by $\frakC_{L, \theta}$. We sometimes omit $L, \theta$ from the notation if the context is clear.
    \item A Stokes graph of $L$ at $\theta$ is a subset of $\frakS_{L, \theta}$.
    \item We say the angle $\theta$ is unobstructed if $\frakC_{L, \theta}$ is empty.
\end{enumerate}
\end{definition}
\begin{remark}
    If one takes the union of the images of Stokes trees, we arrive at a classical point of view including \cite{Virtual} and \cite[\S 9.1]{GMNspec}. However, it is possible that we have a dense/accumulating image. In that case, the classical point of view does not work well.
\end{remark}
In \S \ref{section:proofofunobs}, we state and prove a generic exisntence result of unobstructed angles.

\subsection{Mass filtration}

We next assign a positive real number called BPS mass to each Stokes tree. For each preStokes curve $l$ of type $(ij)$, we set
\begin{equation}
    m(l):=\int_{l}e^{-2\pi i\theta}(\lambda_i-\lambda_j)\in \bR_{\geq 0}.
\end{equation}
For an open Stokes tree $\cT$, we set
\begin{equation}
    m(\cT):=\sum_{l:\text{non-rooted edge of $\cT$}}m(l).
\end{equation}
For the later use, for a point $p\in \cT$ of an open Stokes tree, we set
\begin{equation}
    m(\cT,p):=m(\cT_p)+ \int_{l_p}e^{-2\pi i\theta}(\lambda_i-\lambda_j)
\end{equation}
where $l_p$ is the subset of edge starting from the ascendant vertex ending at $p$, and $\cT_p$ is the subtree of $\cT$ whose rooted edge is $l_p$.

For a closed Stokes tree $\cT$, we set
\begin{equation}
    m(\cT):=\sum_{l:\text{edge of $\cT$}}m(l).
\end{equation}

The Stokes graph $\frakS$ is equipped with an increasing filtration $\lc \frakS_c \rc_{c\in \bR_{> 0}}$ where $\frakS_c$ consists of those with mass less than $c$. We set
\begin{equation}
\begin{split}
    \frakO_c:=\frakS_c\cap \frakO, \frakC_c:=\frakS_c\cap \frakC,
    &\frakS_0:=\bigcap \frakS_c, \frakO_0:=\bigcap_{c\in \bR_{>0}} \frakO_c, \frakC_0:=\bigcap_{c\in \bR_{>0}} \frakC_c.
\end{split}
\end{equation}

Suppose $L$ is strongly GMN. Suppose there are no saddle trajectories at $\theta$. Then any trajectory emanating from a branching value ends at $D$. 
\begin{definition}
    The initial Stokes graph of $L$ is the set of Stokes trees consisting of the trajectories emanating from the turning points. 
\end{definition}
Note that any Stokes tree in the initial graph is mass zero.

\section{The generic existence of unobstructed angles}\label{section:proofofunobs}
\subsection{Statement}

Here is our main theorem.
\begin{theorem}\label{thm:main}
    For a strongly GMN Lagrangian $L$, the set of unobstructed $\theta$'s is dense in $S^1$.
\end{theorem}
We will prove this in the rest of this section.

\subsection{Unobstructed diagram}
We fix a strongly GMN Lagrangian $L$. We also fix $\theta\in S^1$. 

\begin{definition}[Collision]
\begin{enumerate}
\item A collision between two Stokes trees is a crossing of the rooted edges of the trees.
\item We say a set of Stokes trees $\cT_1,...,\cT_s$ forms a collision $p$ if each two of them forms a collision at $p$.
    \item Let $\cT_1,...,\cT_l$ be a set of Stokes trees making a collision at $p$. We say $\cT_1,...,\cT_l$ forms an ordered collision if the types at $p$ of $t_i$ takes the form $(s_i, s_{i+1})$.
    \item We say an ordered collision is cyclic if moreover $s_{l+1}=s_1$. 
\end{enumerate}
\end{definition}
For an ordered collision $p$ formed by $\cT_1,...,\cT_l$, we set
\begin{equation}
    m(\cT_1,...,\cT_l; p):=\sum_{i=1}^lm(\cT_i,p),
\end{equation}
and call it mass of the collision $p$.

\begin{definition}\label{def:unobs} Fix $E\in \bR_{>0}$.
We fix $\theta$. Let $\frakS$ be a Stokes graph. We say $\frakS$ is unobstructed modulo $E$ if 
\begin{enumerate}
\item Locally finite.
    \item Any ordered collision happens outside the set of turning points.
    \item There are no closed Stokes trees with energy less than $E$.
    \item There are no non-discrete overlapping between $(ji)$ and $(ik)$ preStokes curves in Stokes trees in $\frakS$.
    \item The set of ordered collisions with energy less than $E$ is finite in $C\bs D$. 
\end{enumerate}
\end{definition}

\begin{lemma}\label{lem:initialunobs}
Let $L$ be a strongly GMN Lagrangian.
    For $\theta\in S^1$, there exists $\epsilon>0$ such that the initial graph of $L$ is unobstructed for some energy $E>0$ for any $\theta'\in (\theta-\epsilon, \theta+\epsilon)\bs\{\theta\}$.
\end{lemma}
\begin{proof}
By Lemma~\ref{lem:norecurrence}, there exists a nonempty open subset $V\subset S^1$ such that the initial Stokes graph of the associated quadratic differential is locally finite for $\theta\in V$. This is the first condition of Definition~\ref{def:unobs}.

Since the set of turning points and the set of initial Stokes trees are both finite, one can find a nonempty open subset $V'$ in $V$ such that the initial Stokes trees do not pass through turning points for $\theta\in V'$. This is the second condition of Definition~\ref{def:unobs}.

We will treat 3, 4, 5 later.
\end{proof}

\subsection{Deformable Stokes trees}
Since we are interested in the generic behaviour of Stokes trees, we only treat deformable Stokes trees.
We start with then notion of family of preStokes curves.

\begin{definition}
Let $I$ be an interval in $S^1$. Let $J$ be an interval. A map $\gamma\colon I\times J\rightarrow C \bs D$. A type structure of $\gamma$ is an ordered pair $\frakt=(s_1, s_2)$ of continuous sections of $\gamma^*L$.
\end{definition}

\begin{definition}
Let $\gamma\colon I\times J\rightarrow C \bs D$ be an $I$-parametrized family of immersed curves. Let $(s_1, s_2)$ be its type structure.
We say $(\gamma, (s_1, s_2))$ is a family of preStokes curves of type $(s_1, s_2)$ if the restriction to each $\theta\in S^1$ is a preStokes curve for any $\theta\in S^1$.
\end{definition}

Now we can define families of open Stokes trees.
\begin{definition}
    Let $I$ be an interval in $S^1$. A family of Stokes trees is specified by the following:
    \begin{enumerate}
    \item A rooted tree $\cT$.
    \item A continuous map $\gamma_t\colon \cT\times I\rightarrow C\bs D$ such that the restriction to each edge of $\cT$ is a family of preStokes curves.
    \end{enumerate}
\end{definition}

\begin{definition}[Deformable Stokes tree]
A Stokes tree $\cT$ at $\theta$ is deformable if there exists an open neighborhood $I$ of $\theta$ and an $I$-family of Stokes trees whose restriction to $\theta$ is $\cT$.
\end{definition}

\subsection{Inductive construction}
In the following, we will construct unobstructed Stokes diagram inductively. We explain the setup and key lemmas.

\begin{definition}
\begin{enumerate}
    \item     Let $\frakS$ be a Stokes graph. Let $p$ be a non-cyclic ordered collision of open Stokes trees in $\frakS$. We say $p$ is scatterable if there exists an open Stokes tree $\cT$ which contains $p$ as an interior vertex. We say $p$ is unscatterable if it is not scatterable.
    \item We say $p$ is scattered if there exists an open Stokes tree $\cT$ in $\frakS$ which contains $p$ as an interior vertex. We say $p$ is unscattered if it is not scattered.
\end{enumerate}

\end{definition}
\begin{definition}
Let $p$ be a non-cyclic ordered collision of $\cT_1,..., \cT_l$. Suppose the type of $\cT_i$ on $p$ is $(s_i, s_{i+1})$. The one can adjoin a new preStokes curve of type $(s_1, s_{i+1})$ at $p$ to $\cT_1,..., \cT_l$. If the obtained one is a Stokes tree, we call it the scattering of $\cT_1,..., \cT_l$ at $p$.
\end{definition}

\begin{definition}
    Let $\frakS$ be an unobstructed Stokes graph modulo $E$. We set 
    \begin{equation}
        E_\frakS:=\min\lc E, m(\cT_1,..., \cT_l, p)\relmid \text{$p$ is an unscattered ordered collision}\rc.
    \end{equation}
\end{definition}

\begin{definition}
Let $\frakS$ be an $I$-family of Stokes graphs and $E$ be a real positive number. We say $\frakS$ is $E$-scatterble if the following holds: There exists a nonempty relatively compact open subset $J$ of $I$ such that 
\begin{enumerate}
    \item All the cyclic-ordered collisions of mass less than $E$ are scatterable at any $\theta\in J$.
    \item Any scatterings at the collisions with energy less than $E$ is deformable over $J$. 
\end{enumerate}
If $\frakS$ is $E$-scatterble, we define the new family of Stokes graphs $\frakS[E]$ by the union of $\frakS$ and the 
scatterings. 
\end{definition}

We start with $\frakS(0)$, the initial unobstructed diagram, which is a family over a nonempty open subset $I\subset S^1$. 
We will choose an increasing sequence $E_1<E_2<\cdots $ such that $\lim_{n\rightarrow \infty}E_n=\infty$. We inductively set
\begin{equation}
    \frakS(n):=\frakS(n-1)[E_{n-1}]
\end{equation}
when it is defined.
Now we state the key lemma.
\begin{lemma}\label{lemma:keylemma}
 For each $n$, there exists a nonempty open interval $I_n$ on which $\frakS(n)$ is scatterble at $E_n$ and unobstructed modulo $E_n$. Moreover, We can take $I_n$ in a way that
 \begin{enumerate}
     \item $I_n$ is a relatively compact subset of $I_{n-1}$, and
     \item $\lim_{n\rightarrow \infty}E_{\frakS(n)}=+\infty$.
 \end{enumerate}
\end{lemma}

As a corollary of this lemma, we can prove our main theorem.
\begin{proof}[Proof of Theorem~\ref{thm:main}]
Since $I_n$ is a relatively compact subset of $I_{n-1}$, $\bigcap_n I_n$ is not empty (Proof: Take one element from each $I_n$. It forms a bounded sequence. Hence having a convergent subsequence. For each $n$, the limit is contained in $\overline{I_{n-1}}$. Hence it is contained in $I_n$. Hence the limit is in $\bigcap_n I_n$). Since we can take $I_1$ arbitrary close to the initial $\theta$, we have the density result.

Now take $\theta\in \bigcap I_n$. Take any Stokes tree $\cT$. Then $m(\cT)$ is finite, and $\cT\in \frakS(n)$ if $m(\cT)<E_{\frakS(n-1)}$. Then the second part of the above lemma implies that $\bigcup_{n\in\bN} \frakS(n)$ is the desired unobstructed Stokes graph.
\end{proof}

\subsection{Proof of Lemma~\ref{lemma:keylemma}: Induction part}
In this section, we prove the first part of Lemma~\ref{lemma:keylemma}. Since we have already proved for $n=1$ in Lemma~\ref{lem:initialunobs}, we prove the following induction.
\begin{lemma}
    Suppose $\frakS(n)$ is $E_n$-scatterble and unobstructed modulo $E_n$ on an open subset $I_n$ of $I$. Then, for any $E>0$, there exists a relatively compact subset $I_{n+1}$ of $I_n$ such that $\frakS(n+1):=\frakS(n)[E]$ is unobstructed modulo $E$.
\end{lemma}
\begin{proof}
    
We first deal with the generic unobstructedness.
For this, we first show the following finiteness.
    \begin{claim}\label{cl:finitecollision}Let $E$ be a positive number.
        The number of the collisions in $\frakS(n)$ of mass less than $E$ is finite.
    \end{claim}
    \begin{proof}By induction, we can see that the number of involved preStokes curves involved in $\frakS(n)$ is finite. 
        For a $\cT\in \frakS(n)$, take a sequence of points $\{p_i\}$ on the rooted edge such that $\lim_{i\rightarrow \infty}p_i$ going to $D$. 
        By the assumption on the pole order (strongly GMN assumption), we have $\lim_{i\rightarrow \infty}m(\cT, p_i)=\infty$ for any $\cT\in \frakS(n)$. In particular, there exists a neighborhood $U_D$ of $D$ such that 
        \begin{equation}
            \lc p\in C \relmid m(\cT, p)<E, \cT\in \frakS(n)\rc\subset C\bs U_D.
        \end{equation}
        This implies that the collisions of mass less than $E$ is contained in $C\bs U_D$. Since such collisions are discrete by Lemma~\ref{lem:discreteintersection}, we have the desired finiteness.
    \end{proof}
    
    We then add finitely many new Stokes trees. Now we would like to show we can achieve the unobstructedness by perturbing $\theta$. We would like to check from 1 to 5 of Definition~\ref{def:unobs}.

    \begin{enumerate}
\item There are no recurrent trajectories for the lifted quadratic differentials by Lemma~\ref{lem:norecurrence}, and the added Stokes trees are finite by Claim~\ref{cl:finitecollision}. Hence locally finite.
    \item Since the number of turning points are finite by the compactness of $C$, by perturbing slightly, all the added new Stokes trees can avoid turning points. So, any ordered collision happens outside the set of turning points.
    \item By Claim~\ref{cl:finitecollision}, the number of collisions of mass less than $E$ is finite. We will see the absence of cyclically ordered collisions with energy less than $E$ in the following lemma.
    \item Again, by Claim~\ref{cl:finitecollision}, the number of collisions of mass less than $E$ is finite. We can also avoid collisions of preStokes curves of different types by the following lemma.
    \item The set of ordered collisions with energy less than $E$ is discrete in $C\bs D$, since the intersections are discrete by Claim~\ref{cl:finitecollision}.
\end{enumerate}

If there are no cyclic ordered collision, we have done. So, suppose there are several cyclic ordered collisions. The number of them is finite by Claim~\ref{cl:finitecollision}. In this case, we can resolve them by the following claim.

\begin{claim}[Moving lemma]
    Take a point on a smooth part of a Stokes tree in $\frakS_{n,\theta}$. Take a small disk neighborhood. Then the edge divides the disk into two half disks. Then there exists $\epsilon$ such that the corresponding tree in $\frakS_{n,\theta+\eta}$ intersect with the disk on the left half disk for any $0<\eta<\epsilon$.
\end{claim}
\begin{proof}
    We prove by induction. If the tree has no internal vertices, the statement is clear. Suppose $\cT$ be a deformable Stokes tree with some internal vertices. The vertex $p$ next to the root vertex formed by a collision of some $k$ preStokes curves. We denote the corresponding vertex after the perturbation by $\eta$ by $p^\eta$. By the perturbation of $\theta$ by $\eta$, the collision point moves to left with respect to any preStokes curves involved in the collision. In other words, the collision point $p$ satisfies $\Image z_i(p^\eta)<\Image z_i(p)$. Here $z_i$ is the coordinate defined by the integration of the differential defining the corresponding preStokes curve. The root edge curve defines the coordinate $z=\sum z_i$. Hence $\Image z(p^\eta)<\Image z(p)$. Hence the new curve also moves left after the perturbation.
\end{proof}
    Now we can collapse the cases 3 and 4 by the above moving lemma as in Figure~\ref{Figure}.
    This completes the proof of the generic unobstructedness.

    We next show the generic deformability. The following lemma is enough:

\begin{lemma}\label{lem:deform}
    Let $\cT_1,..., \cT_n$ be an $I$-family of Stokes trees. There exists a nonempty open subset $J$ of $I$ satisfying the following: For any $E>0$, the scatterings at the collisions with mass less than $E$ are deformable over $J$.
\end{lemma}
\begin{proof}
The deformability can fail if (1) the collisions form or collapse after perturbations, or (2) the new Stokes curve emanating from the collision does not form a Stokes tree. (1) is OK, since the formation and collapsing of collisions are non-generic phenomena. For (2), it is easy to avoid the situation by using the associated quadratic differential.
\end{proof}
This completes the proof.
\end{proof}

\begin{figure}
\begin{center}
        \includegraphics[scale=0.5]{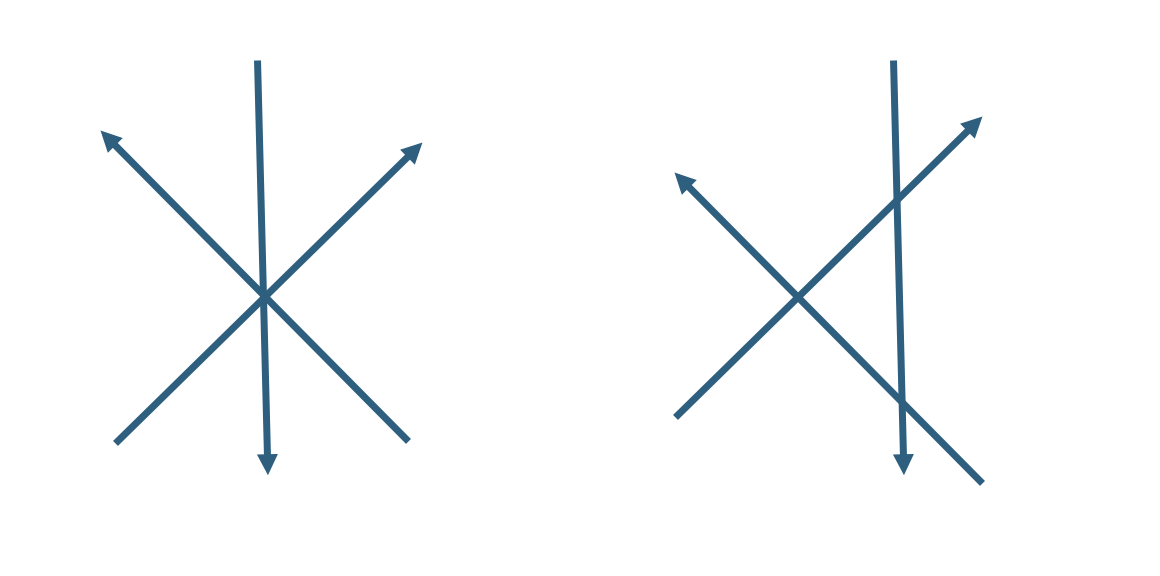}
\end{center}
\caption{Left: Original collision, Right: Perturbed collision} 
\label{Figure}
\end{figure}

\subsection{Proof of Lemma~\ref{lemma:keylemma}: Mass estimate part}
In this section, we prove the latter part of Lemma~\ref{lemma:keylemma}. What we will show is that there exists $m_{\min{}}>0$ which does not depend on $n$ and satisfies $n m_{\min{}}\leq E_{\frakS(n-1)}$. Then the desired statement follows.

We start with some preliminary notions.
For a while, we consider about $\frakS(0)$ paramatrized by $I_0$. In the following discussion, we will discuss uniformly over $\theta\in I_0$. For a branching value $v$ of type $(s_1, s_2)$, let $T$ be the union of the (three) preStokes curves emanating from $v$ of type $(s_1, s_2)$. If it is necessary we retake $I$ sufficiently small, so that we can take a compact neighborhood $U_v$ of $v$ uniformly over $I$  such that 
\begin{enumerate}
    \item $U_v$ does not contain any other branching values, 
    $U_v$ does not intersect with the Stokes trees of type $(i,s_1), (s_2,i)$ in $\frakS(0)$ for any $i$ by 4 of Definition~\ref{def:unobs}.
    \item $U_v\cap T$ is topologically a trivalent tree with a single interior vertex and  $U_v\cap T$ has three boundaries $z_1, z_2, z_3$.
\end{enumerate}

We set 
\begin{equation}
  d(U_v):=\min_i\lc \Re\lb\int_v^{z_i}(\lambda_1-\lambda_2)dz\rb\rc>0
\end{equation}

We also take a small neighborhood $V_v$ of $\overline{U_v\cap T}$ such that the closure of $V_v$ does not contain any other branching values and does not intersect with Stokes trees of $(i,1)$ or $(2,i)$ in $\frakS(0)$. Consider the set $L_{v}$ of connected subsets of preStokes curves of type $(i,1)$ or $(2,i)$ such that each $l\in L_{v}$ starts from a point outside $V_v$ and ends at a point in $U_v\cap T$.
We set
\begin{equation}
    d(V_v):=\inf_{l\in L_{v}}\lc \begin{cases}
    \Re\lb\int_{l}(\lambda_i-\lambda_1)dz\rb&\text{ if $l$ is of type $(i,1)$}\\
    \Re\lb\int_{l}(\lambda_2-\lambda_i)dz\rb&\text{ if $l$ is of type $(2,i)$}
    \end{cases}
    \rc>0.
\end{equation}

Let $w$ be a branching value or an ordered collision of initial Stokes curves which is different from $v$. Let $L_{vw, i, j}$ be the set of  preStokes curves of type $(ij)$ emanating from $w$ passing through $V_v$. We set
\begin{equation}
d_v:=\inf_{w\neq v}\inf_{i,j}\inf_{l\in L_{vw,i,j}}\inf_{x\in V_v}\lc \Re\lb\int_{w}^{x}(\lambda_i-\lambda_j)dz\rb\text{ if $l$ is of type $(ij)$ and the integral is along $l$}\rc.
\end{equation}
This is positive, since the set of turning points and the ordered collisions is discrete and $w$ is outside the closure of $V_v$.

We choose $U_v$ and $V_v$ for each turning point and set 
\begin{equation}
    w_{min}:=\min_v\lc d(U_v), d(V_v), d_v\rc,
\end{equation}
which is again a positive real number. Not that this number is taken uniformly over $I$.

We now set
\begin{equation}
    E_n:=(n+2)w_{min}.
\end{equation}

Now we prove the following:
\begin{lemma}
    We have $(n+1)w_{min}\leq E_{\frakS(n)}$ for any $n$.
\end{lemma}
\begin{proof}
Since $E_{\frakS(0)}\geq d_v\geq w_{min}$, it holds for $n=0$.
    We assume the statement holds true for $n-1$. We classify the unscattered point of $\frakS(n)$ into the followings:
    \begin{enumerate}
    \item Collisions of trees in $\frakS(n-1)$ with mass greater than or equal to $(n+1)w_{min}$.
\item Collisions formed by new Stokes trees in $\frakS(n)$.
    \end{enumerate}
    For the 1st one, the statement is satisfied. We further classify the 2nd one as follows:
    \begin{enumerate}
        \item If a new collision happens outside $U_v$, then the collision carries mass greater than $E_{\frakS(n-1)}+d(U_v)>(n+1)w_{\min{}}$.
        \item Suppose a new collision happens inside $U_v$ between an initial curve and a tree. If the rooted edge of the tree emanates outside $V_v$, it carries mass greater than $E_{\frakS(n-1)}+d(V_v)>(n+1)w_{\min{}}$.
        \item Suppose a new collision happens inside $U_v$ between an initial curve and a tree. If the rooted edge of the tree emanates inside $V_v$, the rooted edge is created by a scattering between a Stokes tree in $\frakS(n-2)$ and a Stokes tree in $\frakS(k)$ with $k\geq 1$. Hence it carries mass greater than $(n-1)w_{min}+(1+1)w_{min}=(n+1)w_{min}$. 
    \end{enumerate}
    This completes the proof.
\end{proof}
This completes the proof of Lemma~\ref{lemma:keylemma}.

\subsection{The finiteness of closed Stokes trees}
As a variant of the above argument, we prove the following theorem.
\begin{theorem}[Gromov compactness]\label{thm:gromovcompact}
For any $\theta\in S^1$ without recurrent trajectories and any $M>0$, the number of closed Stokes trees with mass less than $M$ is finite.
\end{theorem}
We can run the same inductive argument to construct the Stokes graph in the presence of closed Stokes trees. Namely, we only add new preStokes curves to scatterble collisoins. As in the above proof, we obtain the following:
\begin{lemma}
    We have $(n+1)w_{min}\leq E_{\frakS(n)}$ for any $n$.
\end{lemma}

\begin{proof}[Proof of Theorem~\ref{thm:gromovcompact}]
    For any $M$, we can take sufficiently large $n$ such that $E_{\frakS(n)}>M$. Then any closed Stokes tree with mass less than $M$ is contained in $\frakS(n)$. By Claim~\ref{cl:finitecollision}, there are only finitely many closed Stokes trees in $\frakS(n)$. This completes the proof.
\end{proof}

\section{Wall-crossing factors and sheaf quantization}
For a Stokes graph/WKB spectral network, one can assign data of multiplicity, called wall-crossing factors (or BPS-multiplicity). We can use our inductive procedure to construct such data in a canonical way. Moreover, the resulting wall-crossing factors give a sheaf quantization. 
\subsection{Definition}
Let $\bK$ be a field.
Let $M$ be a real manifold and $\bR_t$ be the real line with the standard coordinate $t$. We consider the discrete group of the real numbers $\bR^\delta$ and let it act on $\bR_t$ by addition.

We denote the derived category of equivariant $\bK$-module sheaves by $\Sh^{\bR^\delta}(M\times \bR_t)$. We define the subcategory $\Sh_{\tau\leq 0}^{\bR^\delta}(M\times \bR_t)$ by the microsupport condition $\SS(\cE)\subset \{\tau\leq 0\}$ where $\tau$ is the cotangent coordinate of $\bR_t$. We set
\begin{equation}
     \Sh_{\tau> 0}^{\bR^\delta}(M\times \bR_t):=  \Sh^{\bR^\delta}(M\times \bR_t)/\Sh_{\tau\leq 0}^{\bR^\delta}(M\times \bR_t)
\end{equation}
This is a version in \cite{WKBkuw} of the category introduced by Tamarkin~\cite{Tam}. The feature of this version is the natural enrichment over $\Lambda_0$. 

For an object $\cE\in  \Sh_{\tau> 0}^{\bR^\delta}(M\times \bR_t)$, the positive microsupport $\SS(\cE)\cap \{\tau>0\}$ is well-defined. We set
\begin{equation}
    \mu supp(\cE):=\rho(\SS(\cE)\cap \{\tau>0\})
\end{equation}
where $\rho\colon T^*M\times \lc(t, \tau)\relmid \tau>0\rc\rightarrow T^*M; (x, \xi, t, \tau)\mapsto (x, \xi/\tau)$. 
\begin{definition}
    For a smooth Lagrangian $L\subset T^*C$, we say $\cE\in\Sh_{\tau> 0}^{\bR^\delta}(M\times \bR_t)$ is a sheaf quantization of $L$ if $\musupp(\cE)=L$.
\end{definition}
Given a brane structure $\mathbf{b}$ of $L$, one can further speak about a sheaf quantization of $(L, \mathbf{b})$. Namely, associated to a brane structure of $\mathbf{b}$, we can associate a local system on $L$ to a sheaf quantization $\cE$ of $L$. We say $\cE$ is a sheaf quantization of $(L, \mathbf{b})$ if the local system is the rank 1 constant sheaf.

\begin{conjecture}\label{conj:FukayaaSheaf}
    There exists an infinitesimally wrapped Fukaya category $Fuk(T^*M)$ of nonexact Lagrangians in $T^*M$ with a natural embedding
    \begin{equation}
        Fuk(T^*M)\hookrightarrow \Sh_{\tau> 0}^{\bR^\delta}(M\times \bR_t).
    \end{equation}
\end{conjecture}
This conjecture was checked for integral Lagrangians in \cite{KPS}. 

Now let us assume $M$ be a Riemann surface $C$. Let $L$ be a holomorphic Lagrangian in  $T^*C$. By the Solomon--Verbitsky theorem~\cite{SolomonVerbitsky}, one can consider $e^{2\pi i\theta}\cdot L$ as an object of $Fuk(T^*M)$ for generic $\theta$.
Hence, if the above conjecture holds,
one can deduce the following:
\begin{conjecture}
    For generic $\theta$, the Lagrangian $e^{2\pi i\theta}\cdot L$ is sheaf quantizable.
\end{conjecture}

In the following, we prove this conjecture for spectral curves, without assuming Conjecture~\ref{conj:FukayaaSheaf}.

\subsection{Modification by wall-crossing factors}
For the discussion in this section, we first recall the definition of brane structures.

Let $L$ be a strongly GMN Lagrangian submanifold.
\begin{definition}
    A brane structure of $L$ is a pair of a rank 1 $\bK$-local system $\cL$ on $L$ and a spin structure $\sigma$.
\end{definition}
In the following, we give a spectral-network-friendly form of a brane structure. 

We first note that a spin structure gives a double covering $\widetilde S^*L$ of the co-circle bundle $S^*L$ of $L$. We pull back $\cL$ to $\widetilde S^*L$, then push-forward to $S^*L$. We take the part $\cL^\sigma$ where the fiber monodromy is $-1$. By this procedure, the set of local systems on $L$ and the set of local systems on $S^*L$ whose fiber monodromy is $-1$ are one-to-one, which depends on $\sigma$.

\begin{definition}
A wall-crossing data for a Stokes graph is a $\bK$-valued function $\alpha$ on the set of open Stokes trees in the graph.
\end{definition}
Consider an unobstructed angle $\theta$. Take $c>0$.
Consider wall-crossing data $\alpha$ on $\frakS_c$. We subtract the collision points of $\frakS$ and the branching values and the poles from $C$ and denote it by $C^\circ_\frakS$.

Let $\cT$ be a Stokes tree with the rooted edge of type $(12)$. Let $\gamma$ be a path in $C^\circ_\frakS$ crossing the root edge of $\cT$ at one point. We set $p=\gamma(0), q=\gamma(1)$. We denote the lift of $p$ (resp. $q$) to the sheet $i$ by $p_i$ (resp. $q_i$). Let $\gamma_i$ be the lift of $\gamma$ to sheet (1). We denote the associated Gauss lift by $\widetilde \gamma_i\subset S^*L$.

We associate another path as follows: Let $p'$ be the point where $\gamma$ and $\cT$ meets. For some metric $d$ of $C$, we take an $\epsilon$-neighborhood of $\cT$, and remove the part of the neighborhood ascending to $p'$. Then we take a path gamma detouring the boundary of the neighborhood. We lift this path to the associated Gauss path $\widetilde \gamma_\cT$ on $S^*L$. Note that one can make $\widetilde \gamma_\cT(0)=\widetilde \gamma_1 (0), \widetilde \gamma_\cT(1)=\widetilde \gamma_2 (1)$.

Now we consider the following matrix:
\begin{equation}
 M((\cT, \alpha), \gamma)_c:=
    \begin{pmatrix}
       1&  \alpha_\cT T^{m(\cT)}\cL^\sigma(\widetilde \gamma_\cT)\\
        0&1
    \end{pmatrix}\colon \lb \cL_{\widetilde \gamma_1(0)}\oplus \cL_{\widetilde \gamma_2(0)}\rb \otimes_\bK\Lambda_0/T^c\Lambda_0 \rightarrow \lb  \cL_{\widetilde \gamma_1(1)}\oplus \cL_{\widetilde\gamma_2(1)}\rb\otimes_\bK\Lambda_0/T^c\Lambda_0.
\end{equation}
For an arbitrary path $\gamma$ on $C$, let $\cT_1,..,\cT_n$ be the trees crossing with $\gamma$ in the chronological order. We set
\begin{equation}
M((\frakS_c, \alpha),\gamma):=M((\cT, \alpha_n), \gamma)_cM((\cT, \alpha_{n-1}), \gamma)_c\cdots M((\cT, \alpha_1),\gamma)_c.
\end{equation}
Now we get a new representation of $\Pi_1(C^\circ_\frakS)$ associated to $(\cL, \alpha)$. For any point $p$ which is a branching value or a interior vertex of a tree, we associate $M_p$ by 
\begin{equation}
    M_p:=\lim_{\gamma}M((\frakS_c, \alpha),\gamma)
\end{equation}
where $\gamma$ runs over shrinking circles encircling $p$. 
\begin{definition}
    We say $\alpha$ is compatible if $M_p=\id$ for any such $p$ and any $c>0$.
\end{definition}

\subsection{Initial wall-crossing data}
Let $\frakS_0$ be an unobstructed initial diagram, which consists of Stokes curves emanating from branching values, three for each branching point. 

\begin{definition}
    The canonical initial data assigns $1$ for each Stokes tree (curve) in $\frakS_0$.
\end{definition}

The initial wall-crossing data is indeed compatible at $c=0$. Around a branching point, we can take the following trivialization of the spin structure: For each Stokes curve, when crossing a Stokes curve, we have a transformation of matrix of the trivialization by
    \begin{equation}
         \begin{pmatrix}
        0 & 1 \\
        -1 & 0
    \end{pmatrix}.
    \end{equation}
    Then crossing six times, we have
    \begin{equation}
            \begin{pmatrix}
        0 & 1 \\
        -1 & 0
    \end{pmatrix}^6=
    \begin{pmatrix}
        0 & -1 \\
        -1 & 0
    \end{pmatrix}.
    \end{equation}
This corresponds to one winding around the fiber of $S^*C$. Then 
\begin{equation}
    \begin{pmatrix}
        1 & 1 \\
        0 & 1
    \end{pmatrix}
 \begin{pmatrix}
        0 & 1 \\
        -1 & 0
    \end{pmatrix}
        \begin{pmatrix}
        1 & 1 \\
        0 & 1
    \end{pmatrix}
 \begin{pmatrix}
        0 & 1 \\
        -1 & 0
    \end{pmatrix}
        \begin{pmatrix}
        1 & 1 \\
        0 & 1
    \end{pmatrix}
 \begin{pmatrix}
        0 & 1 \\
        -1 & 0
    \end{pmatrix}
    =    \begin{pmatrix}
        1 & 0 \\
        0 & 1
    \end{pmatrix}.
\end{equation}
Hence it is compatible.

\begin{theorem}\label{thm:constructionofWallcrossing}
    There exists a unique consistent wall-crossing data $\alpha$ whose restriction to $\frakS_0$ is the canonical initial data.
\end{theorem}

\subsection{Inductive construction}
It is enough to describe how we specify wall-crossing data when it is scattered. 
Suppose $p$ is an unscattered collision point. Suppose it consists of Stokes trees $\cT_1,...,\cT_k$ in a cyclic order. We denote the rooted edges of them by $l_1,...,l_k$, where the collision happens. We set
\begin{equation}
    \frakC_p^s:=\lc \lc i_1,..., i_s\rc\subset \lc 1,...,k\rc\relmid l_{i_1},...,l_{i_s}\text{ is an ordered collision}\rc
\end{equation}
for $s\geq 2$.For $\mathbf{i}\in \frakC_p^s$, we put the wall crossing factor $(-1)^{s+1}\prod_{j=1}^s\alpha_{\cT_{i_j}}$ on the scattering $\cT_{\mathbf{i}}$. Then we obtain a new Stokes subgraph $\frakS[p]$ (after perturbing $\theta$ if necessary).
One can directly check the following:
\begin{lemma}
    We have $M(\frakS[p], p)=\id$ modulo $E_{\frakS[p]}$.
\end{lemma}

By repeating this procedure along the construction of the Stokes graph, we complete the proof of Theorem~\ref{thm:constructionofWallcrossing}.

\subsection{Construction of sheaf quantization}

Take an unobstructed angle $\theta$. Fix $E>0$. On each connected component $D$ of the complement of $\frakS$, we consider the projection
\begin{equation}
    \lc (x, t)\in D\times \bR_t\relmid t\geq -\int^x \lambda_i+c\rc\rightarrow D\cong L_i
\end{equation}
where $L_i$ is the connected component of $L\cap T^*D$. Along this morphism, we pull back $\cL$, and denote it by $\cL_{D,i}^c$.

We put the following sheaf
\begin{equation}
S_D^E:=\bigoplus_{i=1}^K\bigoplus_{c\in \bR}\cL_{D,i}^c.
\end{equation}
On each wall and each Stokes tree $\cT$ with the root edge of type $(ij)$ contributing on it, we consider the automorphism
\begin{equation}
    \id+E_{ij}T^{m(\cT)}a(\cT)\in \Aut(S_D^E|_{\text{neighborhood of wall}}).
\end{equation}
By the compatibility, we can glue up them to get a sheaf $S^E$. By taking the inverse limit, we obtain our desired object $S:=\lim_{+\infty\leftarrow E}S^E$.
\begin{theorem}
    The object $S$ is a sheaf quantization of the given brane structure $(L, \mathbf{b})$.
\end{theorem}

\section{Non-abelianization}
\subsection{Novikov ring and equivariant Tamarkin category}
Let $\bR_{\geq 0}$ be the semigroup of real nonnegative numbers. We denote the polynomial ring of $\bR_{\geq 0}$ by $\bK[\bR_{\geq 0}]$. For $a\in \bR_{\geq 0}$, we denote the corresponding indeterminate by $T^a$. We denote the ideal generated by $T^a$ by $\la T^a\ra$. Then we set
\begin{equation}
    \Lambda_0:=\lim_{\substack{\longleftarrow \\ a\rightarrow +\infty}}\bK[\bR_{\geq 0}]/T^a\bK[\bR_{\geq 0}].
\end{equation}
This is called the universal Novikov ring. The unique maximal ideal of $\Lambda_0$ will be denoted by $\Lambda_0^+$. We denote the fraction field by $\Lambda$.

We now recall the main result of \cite{KuwNov}. We denote the derived category of $\Lambda_0$-modules over $C$ by $\Mod(\Lambda_{0C})$. 
\begin{theorem}[\cite{KuwNov}]\label{thm:Novsheaf}
    There exists an almost embedding
    \begin{equation}
        \Sh^{\bR_\delta}_{>0}(C\times\bR_t)\hookrightarrow \Mod(\Lambda_{0C}).
    \end{equation}
\end{theorem}
For a more precise result, we refer to \cite{KuwNov}.

\subsection{Non-abelianization over Novikov field}
It is known/expected that the spectral network transforms a local system on a spectral curve into a local system on $C\bs D$ e.g., \cite{GMNspec, ionita2021spectralnetworksnonabelianization}. This procedure is called {\em non-abelianization}. This is an important topic related to cluster coordinates on character varieties. In this section, we explain how to realize non-abelianization in our setup. 

\begin{theorem}
Let $L$ be a strongly GMN Lagrangian.
    Let $\theta$ be an unobstructed angle. Fix a spin structure $\sigma$ of $L$. Then we have a map
    \begin{equation}
       NA(\sigma): \Loc_1(L,\bK)\rightarrow \Loc_K(C\bs D, \Lambda).
    \end{equation}
    where $\Loc_i(X, A)$ is the set of $A$-local systems on $X$ of rank $i$ for a field $A$.
\end{theorem}
\begin{proof}
In the last section, we constructed a sheaf quantization as an object in $\Sh_{\tau> 0}^{\bR^\delta}(M\times \bR_t)$. By using the main result of Theorem~\ref{thm:Novsheaf}, we can construct a $\Lambda_0$-module sheaves on $C$. By inspecting the construction in \cite{KuwNov}, one can see that, by tensoring $\Lambda$ over $\Lambda_0$, we obtain a locally constant sheaf over $C\bs D$, which is the nonabelianization.
\end{proof}

\begin{remark}
    We conjecture that the resulting local system is covergent under the substitution $T=e^{-1/\hbar}$ for $0<\hbar<<1$ when $\bK=\bC$.
\end{remark}

\section{Exact WKB conjecture}~\label{WKBconjecture}
Exact WKB analysis has been expected for higher order differential equations. However, we cannot find much exact conjectures on exact WKB analysis in the literature. Here, as an application of our results, we give a version of an exact conjecture.

We suppose $\nabla$ be an $\hbar$-flat connection. We denote a WKB formal solution by $\Psi$ and its Borel--Laplace dual by $\cL\Psi$. 

\begin{conjecture}
Assume $\nabla$ is WKB-regular~\cite{WKBkuw} and the spectral curve is strongly GMN. Let $\theta$ be an unobstructed angle whose existence is assured by Theorem~\ref{thm:main}. 
\begin{enumerate}
    \item If $z\in C\bs D$ is not on the Stokes graph at $\theta$, $\cL\Psi$ is analytically continuable on $\bR_{>0}\cdot e^{2\pi i\theta}$. Moreover, $\cL\Psi$ is Laplace transformable in the direction $\theta$.
    \item Let $\cT_1,..,\cT_i,...$ be the set of Stokes trees at $\theta$ passing through $z$. Then the first sheet of the analytic continuation of $\cL\Psi$ is smooth on $\bR_{>0}\cdot e^{2\pi i\theta}\bs \lc m(\cT_i, \theta)\cdot e^{2\pi i\theta}\rc_i$.
\end{enumerate}
\end{conjecture}

\footnotesize
\bibliographystyle{alpha}
\bibliography{bibs.bib}

\noindent
Department of Mathematics, Graduate School of Science, Kyoto University, tatsuki.kuwagaki.a.gmail.com

\end{document}